\documentclass[12pt,reqno]{amsart}
\usepackage{amssymb, amsmath, latexsym, a4}
\usepackage[latin1]{inputenc}
\usepackage{graphicx}
\usepackage[all]{xy}
\usepackage{xcolor}
\usepackage{amsaddr}

\newtheorem{De}{Definition}[section]
\newtheorem{Th}[De]{Theorem}
\newtheorem{Pro}[De]{Proposition}

\DeclareMathOperator{\End}{End}
\DeclareMathOperator{\Ann}{Ann}

\begin{document}

\title{On Algebraic properties of ABO-Blood Type Inheritance Pattern}

\author{J.M. Casas$^1$, M. Ladra$^2$, B.A. Omirov$^3$, R. Turdibaev$^2$}
\address{$^1$Dpto. Matemática Aplicada I, Universidad de Vigo,  E. E. Forestal, Campus Universitario A Xunqueira, 36005 Pontevedra, Spain, jmcasas@uvigo.es}
\address{$^2$ Department of Algebra, University of Santiago de Compostela, 15782, Spain, manuel.ladra@usc.es, rustamtm@yahoo.com}
\address{$^3$Institute of Mathematics, National University of Uzbekistan, Tashkent, 100125, Uzbekistan, omirovb@mail.ru}

\thanks{Corresponding author: manuel.ladra@usc.es, phone: +34 881813138,  fax: +34 881813197}

\begin{abstract}
 We generate an algebra on blood phenotypes with multiplication based on human ABO-blood type inheritance pattern.
  We assume that during meiosis gametes are not chosen randomly. For this algebra we investigate its algebraic properties.
   Namely, the lattice of ideals and the associative enveloping algebra are described.
\end{abstract}

\subjclass[2010]{17D92, 17A20, 92D25}
\keywords{Human ABO-blood type inheritance, blood phenotype,  algebra, absolute nilpotent element, idempotents, ideal, lattice, associative algebra.}

\maketitle

\section{Introduction}

A blood type (also called a blood group) is a classification of blood based on the presence or absence of inherited antigenic
substances on the surface of red blood cells. These antigens may be proteins, carbohydrates, glycoproteins, or glycolipids,
 depending on the blood group system. Some of these antigens are also present on the surface of other types of cells
 of various tissues. Several of these red blood cell surface antigens can stem from one allele (or very closely linked genes)
  and collectively form a blood group system. Human blood group was discovered in 1900 by Karl Landsteiner \cite{Land}.
   Blood types are inherited and represent contributions from both parents.

Distinct molecules called agglutinogens (a type of antigen) are attached to the surface of red blood cells.
 There are two different types of agglutinogens, type ``A'' and type ``B''. Each type has different properties.
  The ABO blood type classification system uses the presence or absence of these molecules to categorize blood
   into four types and is the most important blood-group system in human-blood transfusion. The O in ABO is often called zero.

Establishing the genetics of the ABO blood group system was one of the first breakthroughs in Mendelian genetics.
 There are three alleles or versions of the ABO-blood type genes -  A, B and O.
  The allele O is recessive to A and B, and alleles A and B are co-dominant. It is known that humans are diploid organisms,
   which means that they carry a double set of chromosomes. Therefore blood types are determined by two alleles with six possible combinations:
    AA, BB, OO, AB, OA and OB. Since A and B dominate over O, the genotypes AO and AA express blood type A (phenotype A)
     and BO together with BB correspond to phenotype B.

A number of papers were devoted to study of distribution of blood types in different countries and ethnicities \cite{Chak,Fujita,Kang}.

Some methods for estimating phenotype probabilities for ABO groups were developed and compared in \cite{Sina}. Assuming allele probabilities  to be $p, q$ and $r$ for
the genes A, B, and O, respectively, they obtain some estimates on the probabilities that a  person  has a corresponding phenotype.

Algebraic relations that blood group frequencies satisfy were given in \cite{BernUber}. In \cite{Timur}  a question
 on how the frequencies of human blood genotypes will evolve after several generations in a population has been considered
 and a complete list of all algebraic relations between blood genotypes frequencies is established, as well.

The evolution (or dynamics) of a population comprises a determined change of state in the next generations as a result of reproduction and selection.
 This evolution of a population can be studied by a dynamical system (iterations) of a quadratic stochastic (a so-called evolutionary) operator \cite{Bernstein}.

Most of the numerous papers on the subject were dedicated to cases when during the fertilization parents' gametes are chosen random
and in an independent way. Using Mendel's first law allows to quantify the types of gametes an individual can produce.
 For example, a person with genotype $OA$ during meiosis produces gametes $O$ and $A$ with equal probability $\frac12$.

Throughout the article we assume that child obtains exactly one allele from each parent. However, in \cite{Yamaguchi} and \cite{Yamaguchi2}
 some instances when two alleles $A$ and $B$ were inherited from one parent were described. Therefore, generally speaking pattern of heredity
  of blood types are unpredictable. Using specific models of heredity and collected data that contains these mutations in \cite{Gani} a limiting distribution of blood group is studied.

In this paper we will consider the case when we deviate from Mendelian rules and let some competition for gametes during meioses.
 We will assume that all parents of blood type $A$ and $B$ contribute with gamete $O$ to the child with a constant probability $\alpha$
  and $A$ allele is selected with probability $\beta$ from parents of type $AB$ during meiosis. In case of Mendelian genetics,
   $\beta=\frac12$ and $\alpha=\frac14$. Further, considering the blood phenotypes as independent basis vectors,
     we generate a 4-dimensional vector space over $\mathbb{R}$ and introduce a commutative and non-associative multiplication
      assigning to basis vectors a linear combination of the possible phenotypes of progeny with corresponding probabilities.

After applying some linear basis transformations we obtain an algebra which admits a simpler table of multiplication and
 describe the lattice of ideals of this algebra, and observe that it changes depending on the values of initial parameters.
  Finally, we establish that two distinct such algebras are not isomorphic unless the second parameters of these algebras add up to 1.

\section{Algebras of ABO-Blood type}

Consider blood phenotypes $O, A, B$ and $AB$ as basis elements of a 4-dimensional vector space and a bilinear operation  $\circ$ as the result of meiosis.

In this work we will assume that all parents of blood phenotype $A$ and $B$ have equal probabilities to contribute with allele $O$
 for the child's genotype and we will denote this probability by $p_{O|A}=p_{O|B}=\alpha$. Furthermore, we assume that all parents
  with phenotype (genotype) $AB$ admit equal probabilities to contribute with allele $A$ during meiosis and we denote
   it by $p_{A|AB}=\beta$. Under these assumptions we have the following 10 formal equalities:

\begin{itemize}
\item[(i)] $O\circ O = O$;

\item[(ii)] $O\circ A = p_{O|A} O + (1-p_{O|A})A=\alpha O+(1-\alpha)A$;

\item[(iii)] $O\circ B = p_{O|B} O + (1-p_{O|B})B= \alpha O+(1-\alpha)B$;

\item[(iv)] $O\circ AB = p_{A|AB}A+p_{B|AB}B=\beta A+(1-\beta)B$;

\item[(v)] $A\circ A =p_{O|A}^2 O+ (1-p_{O|A}^2)A=\alpha^2 O + (1-\alpha^2) A$;

\item[(vi)] $\begin{aligned}[t]
A\circ B &= p_{O|A}p_{O|B}O+p_{A|A}p_{O|B}A+p_{O|A}p_{B|B}B+p_{A|A}p_{B|B}AB\\
& = \alpha^2 O+\alpha(1-\alpha)A+\alpha(1-\alpha)B+(1-\alpha)^2AB;
\end{aligned}$

\item[(vii)] $\begin{aligned}[t]
A\circ AB &= p_{A|AB}A +p_{O|A} p_{B|AB}B + p_{A|A}p_{B|AB}AB\\
& = \beta A+\alpha(1-\beta)B+(1-\alpha)(1-\beta)AB;
\end{aligned}$

\item[(viii)] $B\circ B = p_{O|B} ^2 O+ (1-p_{O|B}^2)B=\alpha^2 O+(1-\alpha^2)B$;

\item[(ix)]  $\begin{aligned}[t]
B\circ AB & =p_{O|B} p_{A|AB}A +p_{B|AB}B + p_{B|B}p_{A|AB}AB\\
& = \alpha\beta A+(1-\beta)B+(1-\alpha)\beta AB;
\end{aligned}$

\item[(x)]  $\begin{aligned}[t]
AB\circ AB & =p_{A|AB}^2A+p_{B|AB}^2B+2p_{A|AB}p_{B|AB}AB\\
& = \beta^2 A+(1-\beta)^2B+2\beta(1-\beta)AB.
\end{aligned}$
\end{itemize}

\begin{De} A commutative four-dimensional $\mathbb{R}$-algebra   with basis   $\{O,A,B,AB\}$
 and with multiplication $\circ$ satisfying equalities {\rm(i)--(x)} is called a generalized ABO-blood type algebra (GBTA) and is denoted by $\mathcal{B}(\alpha,\beta)$.
\end{De}

If during meiosis we assume that parents' gametes are chosen random and independently, then $\alpha=\frac14$ and $\beta=\frac12$.

\begin{De}
A generalized ABO-blood type algebra $\mathcal{B}(\frac14,\frac12)$ is called an ABO-blood type algebra (BTA,  short form).
\end{De}

From now on we assume that $0<\alpha, \beta <1$. Note that if we interchange $A$ and $B$ and $\beta$ to $1-\beta = 1-p_{A|AB}=p_{B|AB}$,
 we obtain the same products as above, i.e., we have $\mathcal{B}(\alpha, \beta)\cong \mathcal{B}(\alpha, 1-\beta)$.
  Later in the last section we establish that no other isomorphisms between two GBTAs exist for different values of parameters $\alpha$ and $\beta$.

We can have a look to the algebraic relations defining a GBTA from a different perspective.

Let $x_1,x_2,x_3,x_4$ be corresponding proportions of $O, A, B, AB$ phenotypes in one population. Then for the underlying allele frequencies we have the following equalities
\[p_O= x_1+\alpha x_2+\alpha x_3, \, p_A= (1-\alpha)x_2+\beta x_4,\, p_B= (1-\alpha)x_3+(1-\beta)x_4.\]

Straightforward computation of the frequencies of $O, A, B$ and $AB$ phenotypes in zygotes of the next state yield an extension of Hardy-Weinberg Law:
\[ \left\{ \begin{aligned}
x_1' &=p_O^2\\
x_2' & =p_A^2+2p_A p_O\\
x_3' & =  p_B^2+2p_Bp_O\\
x_4' &=2p_Ap_B.
\end{aligned} \right. \]

Consider $\mathbb{S}^3=\{\mathbf{x}=(x_1,x_2,x_3,x_4)\in \mathbb{R}^4 \mid x_1+x_2+x_3+x_4=1,  x_i \geq 0,\, 1\leq i \leq 4\}$
a 3-dimensional canonical simplex. Following \cite{Lyubich}, we have a so-called evolutionary (quadratic stochastic) operator
 $V \colon \mathbb{S}^3\to \mathbb{S}^3$  describing an evolution of the population mapping a state $\mathbf{x}=(x_1,x_2,x_3,x_4)$
 to the next state $V(\mathbf{x})=(x_1',x_2',x_3',x_4')$. By linearity $V$ is extended to $\mathbb{R}^4$ if it is necessary.

The relation that establishes a connection between evolutionary operator $V$ and multiplication $\circ$ of a GBTA is
 $\mathbf{x}\circ \mathbf{x} = V(\mathbf{x})$ and consequently \[\mathbf{x}\circ \mathbf{y} = \frac{1}{4}(V(\mathbf{x+y})-V(\mathbf{x-y})).\]

In order to simplify our investigation of the structure of  a GBTA we make the following linear basis transformation
\[\left\{ \begin{aligned}
o&=O\\
a&= \frac{1}{(1-\alpha)^2}\cdot (O-A)\\
b&= \frac{1}{(1-\alpha)^2}\cdot (O-B)\\
ab&= \frac{1}{(1-\alpha)^3}\cdot (\alpha O-\beta A - (1-\beta)B+(1-\alpha)AB)
\end{aligned} \right.
\]
and obtain a simpler table of multiplication of  a GBTA:

\[\mathcal{B'}(\lambda, \beta): \left\{ \begin{aligned}
{o\circ o}&=o\\
{o\circ a}&= \lambda a\\
{o \circ b}&= \lambda b\\
{a \circ a}&= a\\
{b \circ b}&= b\\
{a \circ b}&= \frac{\lambda-\beta}{\lambda}\cdot a + \frac{\lambda-(1-\beta)}{\lambda} \cdot b + ab,
\end{aligned} \right.
\]

where $\lambda=1-\alpha$.

Due to convenience of the above products further we will investigate algebraic properties of the algebra $\mathcal{B'}(\lambda, \beta)$.

\section{Absolute nilpotent and idempotent elements}

Taking an initial point $\mathbf{x}\in \mathbb{S}^3$ one can consider its trajectory $\{V^k(\mathbf{x})  \mid k\geq 1\}$.
 The study of limit behavior of trajectories of quadratic stochastic operators played an important role in several questions of population genetics.
  Trajectories of genotype frequencies were studied in \cite{Jamilov,Lyubich2}.
   Note that, if a limit point of a trajectory exists then it is a fixed point, that is $\mathbf{x}=V(\mathbf{x})=\mathbf{x}\circ \mathbf{x}$.

\begin{De}
An element $x$ of an algebra $(A, \circ)$ with $x \circ x =\mu x$ is said to be absolute nilpotent if $\mu=0$ and idempotent if $\mu =1$.
\end{De}

We see that set of absolute nilpotent elements of a GBTA constitutes the kernel of $V$ and idempotent elements are fixed points of $V$.

It is easy to see that $ab$ is annihilated in the  algebra $\mathcal{B'}(\lambda, \beta)$ and it is absolute nilpotent element, while $o,\, a $ and $b$ are idempotent elements.

\begin{Th} \label{absolutely nilpotent}
The set of absolute nilpotent elements of $\mathcal{B'}(\lambda, \beta)$ is $\langle ab\rangle$.
\end{Th}
\begin{proof}
 If $\alpha o +n$, where $n$ belongs to the ideal $\langle a, b, ab\rangle$, is an absolute nilpotent element, then
$(\alpha o +n)^2=\alpha^2 o + 2\alpha o \circ n + q\circ n \equiv \alpha^2 o  \ (\bmod \ \langle a, b, ab\rangle)$ implies  $\alpha=0$.

Let $n=xa+yb+zab$ be an absolute nilpotent. Then
\begin{align*}
n\circ n &= x^2a+y^2b+2xy a\circ b\\
& = \left(x^2+\frac{2xy}{\lambda}(\lambda-\beta)\right)a +\left(y^2+\frac{2xy}{\lambda}(\lambda+\beta-1)\right)b+2xyab=0.
\end{align*}

Hence, we need to solve the system of equations
\[
\left\{ \begin{aligned}
 0&= x^2+ \frac{2xy}{\lambda}(\lambda-\beta)\\
0&= y^2+ \frac{2xy}{\lambda}(\lambda+\beta-1) \\
0&=2xy
\end{aligned}  \right.
\]
Solution is obviously any triple $(0,0,z)$.
\end{proof}

Now we describe the idempotents of $\mathcal{B'}(\lambda, \beta)$.

Let us denote by $P:=\{(\lambda,\beta)  \mid 0<\lambda
\leq \frac13,\, \beta = \frac12\left( 1\pm \sqrt{(1-\lambda)(1-3\lambda)} \right) \}$.

\begin{Th} \label{idempotent}
	For the algebra $\mathcal{B'}(\lambda,\beta)$, with  $(\lambda,\beta) \in P$, the set of idempotents is
\[\{o , \ a, \  b, \ o +(1-2\lambda)a, \ o +(1-2\lambda)b\}.\]
	
	For the algebra $\mathcal{B'}(\lambda,\beta)$, with  $(\lambda,\beta) \not\in P$, the set of idempotents is
\[\{o , \ a, \  b,  \ o +(1-2\lambda)a, \  o +(1-2\lambda)b, \ j_0, \ j_1\},\]
	where
\begin{multline*} j_{\alpha}= \alpha o +\frac{\lambda (1-2\alpha \lambda)}{-3\lambda^2+4\beta^2+4\lambda-4\beta} (2\beta-\lambda)a+
 \frac{\lambda (1-2\alpha \lambda)}{-3\lambda^2+4\beta^2+4\lambda-4\beta} (2-2\beta-\lambda)b\\
  +2 \left(\frac{\lambda(1-2\alpha \lambda)}{-3\lambda^2+4\beta^2+4\lambda-4\beta}\right)^2 (2\beta-\lambda) (2-2\beta-\lambda)ab,
  \quad \alpha=0, 1.
 \end{multline*}
\end{Th}

\begin{proof}
  Similarly as in the proof of Theorem~\ref{absolutely nilpotent}, we deduce that an idempotent admits the form $i=\alpha o  + xa+yb+zab$ with $\alpha=\alpha^2$.

Equality $i\circ i =i $ yields the system of equations
\[ \left\{ \begin{array}{lcl}
x&=& x^2+ \frac{2xy}{\lambda}(\lambda-\beta)+2\alpha\lambda x \\
y&=& y^2+ \frac{2xy}{\lambda}(\lambda+\beta-1)+2\alpha\lambda y\\
z&=& 2xy.
\end{array} \right.  \]

The first two equations transform into
\[\left\{\begin{array}{lcl}
0&=& x\left(x+ 2y\frac{\lambda-\beta}{\lambda}+(2\alpha\lambda-1)\right) \\
0&=& y\left(y+ 2x\frac{\lambda+\beta-1}{\lambda}+(2\alpha\lambda-1) \right).
\end{array} \right.  \]

If $x=0$, then $z=0$ and we obtain $y(y+(2\alpha \lambda -1))=0$. Further, either $y=0$ or $y=1-2\alpha\lambda$.
This yields two idempotents  $\alpha o $ and $\alpha o +(1-2\alpha\lambda)b$.
Taking into account the possible values for $\alpha$ we conclude that $o ,b, o +(1-2\lambda)b$ are idempotents.

If $y=0$, then $z=0$ and we obtain $x(x+ 2y\frac{\lambda-\beta}{\lambda}+(2\alpha\lambda-1))=0$.
Further, either $x=0$ or $x=1-2\alpha\lambda$. This yields  the idempotents $\alpha o $ and $\alpha o +(1-2\alpha\lambda)a$.
Since $\alpha =0 ,\, 1$ we get that $a$ and $o +(1-2\lambda)a$ are idempotents.

Now we consider the case when $xy\neq 0$. We obtain
\begin{equation}\label{E:ec1}
\left\{ \begin{aligned}
x+ 2 \frac{\lambda-\beta}{\lambda}y &=  1-2\alpha\lambda\\
2 \frac{\lambda+\beta-1}{\lambda}x+y &= 1-2\alpha\lambda.
\end{aligned} \right.
\end{equation}

In order to solve the system we consider the following two cases:

\textbf{Case 1.} Let $\displaystyle \det \begin{pmatrix}
1 & 2 \frac{\lambda-\beta}{\lambda}\\
2 \frac{\lambda+\beta-1}{\lambda} & 1 \\
\end{pmatrix} \neq 0$, i.e., $1-\frac{4}{\lambda^2} (\lambda-\beta)(\lambda+\beta-1)\neq 0$. Then we  have the unique solution of \eqref{E:ec1}
\begin{align*}
x&=\frac{\lambda (1-2\alpha \lambda)}{-3\lambda^2+4\beta^2+4\lambda-4\beta} (2\beta-\lambda),\\
y&=\frac{\lambda (1-2\alpha \lambda)}{-3\lambda^2+4\beta^2+4\lambda-4\beta} (2-2\beta-\lambda).
\end{align*}
Hence $z=2xy=2 \displaystyle \left(\frac{\lambda (1-2\alpha \lambda)}{-3\lambda^2+4\beta^2+4\lambda-4\beta}\right)^2 (2\beta-\lambda)
 (2-2\beta-\lambda)$. Thus  the desired idempotents are  $\alpha o  +x a + yb +z ab$, where $\alpha=0, \, 1$.
\medskip

\textbf{Case 2.} Let $\displaystyle\det\left( \begin{array}{cc}
1 & 2 \frac{\lambda-\beta}{\lambda}\\
2 \frac{\lambda+\beta-1}{\lambda} & 1 \\
\end{array}\right) =0$. This condition is equivalent to $(1-3\lambda)(1-\lambda)=(2\beta-1)^2$. For this condition to hold it
 is necessary and sufficient that $0<\lambda \leq \frac13$ and $\beta = \frac12\left(1\pm\sqrt{(3\lambda-1)(\lambda-1)}\right)$.
  So the determinant is zero if and only if $(\lambda, \beta ) \in P$.

It is easy to see that one obtains the first equality by multiplying the second one to $2\cdot \frac{\lambda-\beta}{\lambda}$.
 Thus we have  $2\cdot \frac{\lambda-\beta}{\lambda}(1-2\alpha \lambda)=1-2\alpha \lambda$,
   consequently  either $2\cdot \frac{\lambda-\beta}{\lambda}=1$ or $1-2\alpha \lambda=0$. Simple observations derive to contradiction with $(\lambda, \beta ) \in P$.
\end{proof}

\section{Plenary powers}
In this section we investigate which states $\mathbf{x}\in \mathbb{R}^4$ will admit zero as a limit point after a finite number of iterations,
 i.e., $V^k(\mathbf{x})=0$ for some $k\geq 1$.  Recall from the previous section that  the kernel of $V$  consists of the absolute nilpotent elements.

\begin{De}
For an arbitrary  element $m$ in $\mathcal{B'}(\lambda, \beta)$, its so called  plenary powers are defined recursively by
\[m^{[1]}=m, \qquad m^{[n+1]}=m^{[n]}\circ m^{[n]}, \quad n \geq 1.\]

An element $m$ in $\mathcal{B'}(\lambda, \beta)$ is said to be solvable if there exists $n \in \mathbb{N}$ such that  $m^{[n]}=0$ and the least such number $n$,  its solvability index.
\end{De}

Obviously, absolute nilpotent elements are solvable with index of solvability equal to 2.

\begin{Th}
	For an algebra of ABO-blood type $\mathcal{B}'(\lambda, \beta)$ to admit a solvable element of index $n\geq 3$ it is necessary and sufficient that $(\lambda,\beta)\in P$.
	
	Moreover, solvable elements of degree $n$ are
	\[-2^{n-4}\left( \frac{\lambda+\beta-1}{\lambda}\right)^{n-4} t  a +t  b + s  ab, \ \text{ where } \ t,s \in \mathbb{R}, \ t\neq 0.\]
\end{Th}

\begin{proof}
Let $m$ be a solvable element  in $\mathcal{B'}(\lambda, \beta)$ with solvability index $n$.
  Similar as in the proof of Theorem~\ref{absolutely nilpotent} we can assume that a solvable element does not contain $o$ component.

Set $m^{[k]}=X_ka+Y_k b + Z_k ab$, for all $1\leq k \leq n$.

In order to obtain recursive relations between the pairs  $(X_{k+1},Y_{k+1})$ and $(X_{k},Y_{k})$, for $1\leq k \leq n-1$,
  we consider  the following equivalences modulo the ideal $\langle ab\rangle$:
\[\begin{array}{rl}
X_{k+1}a+Y_{k+1}b  \equiv & m ^{[k+1]}= m ^{[k]}\circ m ^{[k]}\\
\equiv& (X_{k}a+Y_{k}b)\circ (X_{k}a+Y_{k}b)\\
\equiv & X_{k}\left(X_{k}+2 \frac{\lambda-\beta}{\lambda}Y_{k}\right) a + Y_{k}\left( 2 \frac{\lambda+\beta-1}{\lambda}X_{k}+Y_{k}\right)b. \end{array}\]
Therefore we obtain a system of equations
\[S_{k}: \left\{\begin{array}{lcl}
X_{k+1}&=& X_{k}\left(X_{k}+ 2\frac{\lambda-\beta}{\lambda}Y_{k}\right) \\
Y_{k+1}&=& Y_{k}\left(2\frac{\lambda+\beta-1}{\lambda}X_{k}+Y_{k}\right).
\end{array} \right.\]

Since we have described all absolute nilpotent elements, we assume $n\geq 3$.
Due to $m ^{[n-1]}$ is an absolute nilpotent element, we know that $m^{[n-1]} \equiv 0 \ (\bmod \ \langle ab \rangle)$, so $X_{n-1}=Y_{n-1}=0$.
 Therefore, the system $S_{n-2}$ has the form
\[\left\{\begin{array}{lcl}
0&=& X_{n-2}\left(X_{n-2}+ 2\frac{\lambda-\beta}{\lambda}Y_{n-2}\right) \\
0&=& Y_{n-2}\left(2\frac{\lambda+\beta-1}{\lambda}X_{n-2}+Y_{n-2}\right).
\end{array} \right. \]

Obviously, either $X_{n-2}Y_{n-2}\neq 0$ or  $X_{n-2}=Y_{n-2}=0$. But in the last case $m^{[n-2]}=Z_{n-2}ab$ and
 $m ^{[n-1]}=0$ which is a contradiction. Therefore,  $X_{n-2}Y_{n-2}\neq 0$ and we have the linear system of equations:

\[\left\{\begin{array}{lcl}
0&=& X_{n-2}+ 2\frac{\lambda-\beta}{\lambda}Y_{n-2}\\
0&=& 2\frac{\lambda+\beta-1}{\lambda}X_{n-2}+Y_{n-2}.
\end{array} \right.\]

If the determinant of this system is not zero then  we obtain a trivial solution which contradicts our assumption.
 Therefore, the determinant is zero, i.e., $ (\lambda, \beta)\in P$ and we get a solution
 $(X_{n-2},Y_{n-2})=(-2\cdot\frac{\lambda-\beta}{\lambda}t, t)$ for some $t\in \mathbb{R}^*$.

Let $n=3$. Then an element $m=X_1a+Y_1b+Z_1ab$ is solvable if and only if the following holds:
\begin{enumerate}
\item[1.] $(\lambda, \beta)\in P$.

\item[2.] $\displaystyle(X_1,Y_1,Z_1)=(-2 \frac{\lambda-\beta}{\lambda} t, t, s)$ for free parameters $t$ and $s$, where $t\neq 0$.
\end{enumerate}

Now let us assume $n\geq 4$.

Using singularity of the determinant we can transform the last equality in the system $S_{k}$ to the following form

\[\left\{\begin{array}{lcr}
 X_{k+1}&=& X_{k}\left(X_{k}+ 2\frac{\lambda-\beta}{\lambda}Y_{k}\right) \\
 Y_{k+1}&=& 2\frac{\lambda+\beta-1}{\lambda}Y_{k}\left(X_{k}+2\frac{\lambda-\beta}{\lambda}Y_{k}\right)
 \end{array} \right.\] for any $1\leq k \leq n-1$.

 Recall that we are in the case $X_{k}Y_{k}\neq 0$ for $1\leq k \leq n-2$ (otherwise, the index of solvability is less than $n$).

 Since both sides of each one of the equalities in the above system are assumed to be non-zero for $1\leq k \leq n-3$, we obtain
 \[\frac{X_{k}}{Y_{k}}=2\cdot \frac{\lambda+\beta-1}{\lambda}\cdot \frac{X_{k+1}}{Y_{k+1}},\textrm{ for any } 1\leq k \leq n-3.\]

 Therefore,
 \[\frac{X_1}{Y_1}=\left(2 \frac{\lambda+\beta-1}{\lambda}\right)^{n-3} \frac{X_{n-2}}{Y_{n-2}}
 =\left(2 \frac{\lambda+\beta-1}{\lambda}\right)^{n-3} \left(-2 \frac{\lambda-\beta}{\lambda}\right)=-\left(2 \frac{\lambda+\beta-1}{\lambda}\right)^{n-4}.\]

Hence, for an element $m=Xa+Yb+Zab$ to be solvable with index of solvability $n\geq 4$ it is necessary and sufficient the following  to hold:
\begin{enumerate}
\item[1.] $(\lambda,\beta)\in P$.

\item[2.] $\displaystyle(X,Y,Z)=\bigg(-2^{n-4}\Big( \frac{\lambda+\beta-1}{\lambda}\Big)^{n-4} t, t, s\bigg)$ where $t,s \in \mathbb{R}, \,t\neq 0$.
\end{enumerate}

In fact, if $n=3$, then  $-2^{n-4}\left( \frac{\lambda+\beta-1}{\lambda}\right)^{n-4}=\left(-2\cdot
 \frac{\lambda+\beta-1}{\lambda}\right)^{-1}$ $= - 2\cdot \frac{\lambda-\beta}{\lambda}$. \end{proof}

\section{Ideals of $\mathcal{B}'(\lambda,\beta)$}

In this section we will find all ideals of $\mathcal{B}'(\lambda,\beta)$. The lattice of ideals will depend on values that parameters $\lambda$ and $\beta$ take.

\begin{Pro}\label{max}
The ideal $\langle a,b, ab\rangle$ is the only maximal ideal of $\mathcal{B}'(\lambda,\beta)$.
\end{Pro}

\begin{proof}
 Let $X=x_1 o +x_2a+x_3b+x_4ab$ be an element of an ideal $I$ of the algebra $\mathcal{B}'(\lambda, \beta)$.
Then $X\circ o  \in I$ implies $x_1 o +\lambda x_2 a + \lambda x_3 b \in I$.

Considering $\lambda X - X\circ o  \in I$ yields $(\lambda-1)x_1 o  +x_4 ab \in I$.
Multiplying the last one by $o $ we obtain $(\lambda -1)x_1 o \in I$. Since $\lambda\neq 1$ we get $x_1 o  \in I$ and
 therefore, $x_1a=x_1 o  \circ \frac1{\lambda}a\in I$ and $x_1b=x_1 o \circ \frac1{\lambda}b\in I$. If $x_1\neq 0$,
  then $o ,a,b \in I$ and $ab\in I$ which yields $I=\mathcal{B}'(\lambda,\beta)$. Therefore, $X=x_2a+x_3b+x_4 ab$
   and $I\subseteq \langle a,b,ab \rangle$ which is a maximal ideal.
\end{proof}

Let us focus our attention on 2-dimensional ideals.
\begin{Pro}
	The algebra $\mathcal{B}'(\lambda, \beta)$, where $\beta\neq \lambda, \beta\neq 1-\lambda$ does not admit 2-dimensional ideals.
 In otherwise $\langle a,ab\rangle, \, \langle b,ab \rangle  \unlhd \mathcal{B}'(\frac12, \frac12)$,
	$\langle a,ab\rangle \unlhd \mathcal{B}'(\lambda, 1-\lambda)$ and
	$\langle b,ab\rangle \unlhd \mathcal{B}'(\lambda, \lambda)$, where $\lambda\neq \frac12$.
\end{Pro}

\begin{proof}
 Let $I$ be a 2-dimensional ideal of the algebra. Proposition~\ref{max} yields that any element $X\in I$ is of the form $X=x_2a+x_3b+x_4ab$.
   From $X-\frac1{\lambda}X\circ o \in I$ we have $x_2a+x_3b, \, x_4 ab \in I$.

The following belongings hold:
\begin{align*}
&x_2^2a -x_3^2 b  =(x_2 a+x_3 b)\circ (x_2a-x_3b) \in I,\\
&(x_2+x_3)a\circ b  =(x_2 a+x_3 b)\circ (a+b) - (x_2 a+x_3 b) \in I,\\
&(x_2+x_3)\frac{\lambda-\beta}{\lambda}a  =
(x_2+x_3)(a\circ b )\circ a-(x_2+x_3)\frac{\lambda-\beta+1}{\lambda}a\circ b\in I,\\
 &(x_2+x_3)\frac{\lambda-\beta+1}{\lambda}b  = (x_2+x_3)(a\circ b )\circ b-(x_2+x_3)\frac{\lambda-\beta}{\lambda}a\circ b\in I,\\
&(x_2+x_3)ab  =(x_2+x_3)a\circ b - (x_2+x_3)\frac{\lambda-\beta}{\lambda}a - (x_2+x_3)\frac{\lambda-\beta+1}{\lambda}b\in I
\end{align*}

In order to complete the description of 2-dimensional ideals we consider distinctive cases.

\begin{enumerate}
	\item[\textbf{Case 1.}] Let $\lambda \neq \beta$ and $\lambda \neq \beta -1$.
	
	The above belongings imply $(x_2+x_3)a$, $(x_2+x_3)b, (x_2+x_3)ab\in I$. Consequently, $x_2+x_3=0$ and $X=x_2(a-b)+x_4ab$.
	
	Consider  $(\lambda-1)b=\lambda(a-b)\circ b -\lambda ab-(\lambda-\beta) (a - b) \in I$.
 Thus $b\in I$  which gives $a\in I$ and we derive into a contradiction. Therefore, in this case there are no 2-dimensional ideals.

	\item[\textbf{Case 2.}] Let $\lambda=\beta$.

	Then $a\circ b = \frac{2\lambda-1}{\lambda}b+ab$. Hence $(2\lambda-1)(x_2+x_3)b= (x_2+x_3)(a\circ b)\circ o \in I$ and we continue by considering the following subcases:
	
	\begin{enumerate}
		\item[\textbf{Case 2.1}] Let $\lambda\neq \frac12$.
		
		Then $(x_2+x_3)b\in I,\, x_2(a-b)=x_2a+x_3b- (x_2+x_3)b\in I$.
		Together with $x_4ab\in I$ we need to have $(x_2+x_3)x_2x_4=0$. Then the following subcases  occur:
		\begin{enumerate}
			\item[\textbf{Case 2.1.1.}] Let $x_4=0$.
			
			Then any element in the ideal is in the form $X=x_2a+x_3b$. But $ab=a\circ b -\frac{2\lambda-1}{\lambda}b\in I$, which is a contradiction.
			
			\item[\textbf{Case 2.1.2.}] Let $x_2=0$.
			
			Then every element in the ideal is in the form $X=x_3b+x_4ab$ and $I=\langle b, ab\rangle$ is an ideal.
			
			\item[\textbf{Case 2.1.3}] Let $x_2+x_3=0$.
			
			Then $a-b, ab \in I$. However $\frac{\lambda-1}{\lambda}b=(a-b)\circ b -ab \in I$ and we obtain $a,b \in I$,  which is a contradiction.
		\end{enumerate}

		\item[\textbf{Case 2.2}] Let $\lambda=\frac12$. Then $\beta=1-\beta=\frac12$ and $a\circ b= ab$.
		
		\begin{enumerate}
			\item[\textbf{Case 2.2.1}] Let $x_2+x_3=0$.
			
			Then $a-b, ab \in I$, while $a=(a-b)\circ a+ab\in I$ and therefore $b\in I$, which is a contradiction.
			
			\item[\textbf{Case 2.2.2}] Let $x_2+x_3\neq 0$.
			
			Then $ab\in I$.
			Note that $x_2(x_2+x_3)a=x_2^2a-x_3^2b+x_3(x_2a+x_3b)\in I$ and similarly $x_3(x_2+x_3)b\in I$. Therefore, $x_2a,\, x_3b\in I$. Since $I$ is 2-dimensional then  $x_2x_3=0$.
			\begin{enumerate}
				\item[\textbf{Case 2.2.2.1}] Let $x_3=0$.
				
				Then $I=\langle a, ab\rangle$ is an ideal.
				
				\item[\textbf{Case 2.2.2.2}] Let $x_2=0$.
				
				Then $\langle b,ab \rangle$ forms an ideal.
			\end{enumerate}
		\end{enumerate}
	\end{enumerate}
	\item[\textbf{Case 3.}] Let $\lambda=1-\beta$.
	
	The study of this case is carried out analogously to the second case and gives same results up to substitution of $a$ to $b$ and vice versa.
\end{enumerate}
\end{proof}

Now we analyze 1-dimensional ideals.

\begin{Pro}
	The ideal $\langle ab \rangle$ is the only 1-dimensional ideal of $ \mathcal{B}'(\lambda,\beta)$.
\end{Pro}

\begin{proof}
 Let $I=\langle x \rangle$ be an ideal of $\mathcal{B}'(\lambda, \beta)$. If $x\circ x =0$ then $x$ is an absolute nilpotent element
  and we know that in this case $I=\langle ab \rangle$ (see Theorem~\ref{absolutely nilpotent}).

If $x\circ x =\delta x$ for a non-zero $\delta$, then denoting $y=\frac1{\delta}x$ yields $y\circ y=y$, i.e.,
 $I$ is generated by an idempotent. The proof of the proposition is completed by checking which idempotents from Theorem~\ref{idempotent} generate a 1-dimensional ideal.
\end{proof}

Summarizing the above results we state the following

\begin{Th}
	The lattices of ideals of corresponding algebras are:
	
	
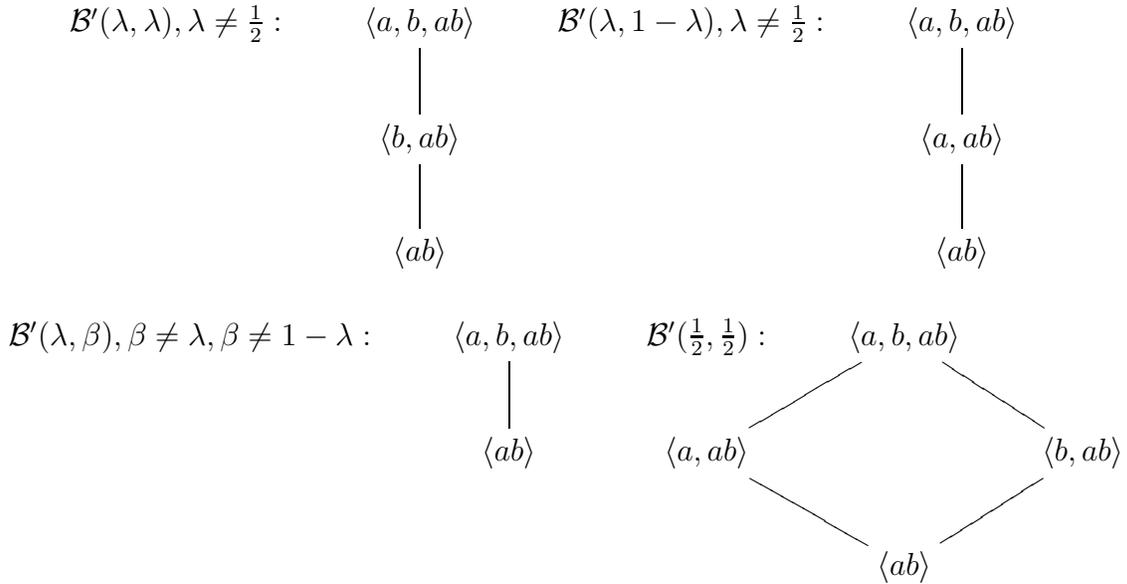
\[\xymatrix{\mathcal{B}'(\lambda, \lambda), \lambda \neq \frac{1}{2}:  & \langle a, b, ab  \rangle \ar@{-}[d] & \mathcal{B}'(\lambda, 1- \lambda),
 \lambda \neq \frac{1}{2}: & \langle a, b, ab  \rangle \ar@{-}[d]\\
		& \langle b, ab  \rangle \ar@{-}[d]& & \langle a, ab  \rangle \ar@{-}[d] \\
		& \langle ab  \rangle &  & \langle ab  \rangle
	}\]
	
	\[\xymatrix{\mathcal{B}'(\lambda, \beta),   \beta \neq \lambda, \beta \neq 1- \lambda:  & \langle a, b, ab  \rangle \ar@{-}[d] &
 \mathcal{B}'(\frac{1}{2}, \frac{1}{2}): & \langle a, b, ab  \rangle \ar@{-}[ld] \ar@{-}[rd]\\
		& \langle  ab \rangle &  \langle a, ab  \rangle \ar@{-}[rd] & & \langle b, ab  \rangle \ar@{-}[ld] \\
		&  &  & \langle ab  \rangle
	}\]
	
\end{Th}

\section{Associative Enveloping Algebra of $\mathcal{B'}(\lambda,\beta)$}

For an arbitrary algebra $A$ we can consider its embedding in  the associative algebra $\End (A)$ via left and right actions of $A$ on $A$.

Consider the  operators of left multiplication by basis elements of the algebra $\mathcal{B'}(\lambda,\beta)$. Their matrix forms are the following:

\[l_{o }=\left(\begin{array}{cccc}
1 & 0 & 0 & 0\\
0 &\lambda &0 & 0 \\
0 &0 &\lambda &0 \\
0 &0 &0 &0
\end{array}\right)=E_{11}+\lambda E_{22}+\lambda E_{33},\]
\[l_{a}= \left(\begin{array}{cccc}
0 & 0 & 0 & 0\\
\lambda & 1 & \frac{\lambda - \beta}{\lambda}& 0\\
0 &0 & \frac{\lambda+\beta-1}{\lambda} &0 \\
0 &0 &1 &0 \\
\end{array}\right) =\lambda E_{21}+E_{22}+ \frac{\lambda - \beta}{\lambda}E_{23}+ \frac{\lambda+\beta-1}{\lambda} E_{33}+E_{43},\]
\[l_{b}= \left(\begin{array}{cccc} 0 & 0 & 0 & 0\\ 0&  \frac{\lambda - \beta}{\lambda} &0 &0 \\
\lambda & \frac{\lambda+\beta-1}{\lambda} & 1 &  0\\ 0 &1 &0  &0 \\\end{array}\right)=
 \frac{\lambda - \beta}{\lambda} E_{22}+\lambda E_{31}+\frac{\lambda+\beta-1}{\lambda} E_{32}+E_{33}+E_{42}.\]
\[l_{ab}=O _4.\]

Let $\mathcal{A}$  be the associative subalgebra of the algebra $\End\big(\mathcal{B'}(\lambda,\beta)\big)$ with the generating set $\{l_o,l_a, l_b \}$.

We denote some subalgebras of the matrix algebra $M_4(\mathbb{R})$ as follows:
\[M_0=\left\langle
\begin{pmatrix}
* & 0 & 0 & 0\\
* & *& 0& 0\\
* & 0& *& 0\\
* &* & *& 0\\
\end{pmatrix}
\right\rangle, \qquad
M_1=\left\langle
\begin{pmatrix}
* & 0 & 0 & 0\\
* & *& *& 0\\
* & 0& *& 0\\
* &* & *& 0\\
\end{pmatrix}
\right\rangle,\]
\[M_2=\left\langle
\begin{pmatrix}
* & 0 & 0 & 0\\
* & *& 0& 0\\
* & *& *& 0\\
* &* & *& 0\\
\end{pmatrix}
\right\rangle, \qquad
M_3=\left\langle
\begin{pmatrix}
	* & 0 & 0 & 0\\
	* & *& *& 0\\
	* & *& *& 0\\
	* &* & *& 0\\
\end{pmatrix}
\right\rangle.\]

Note that the  generators of $\mathcal{A}$ are contained in $M_3$. Below,
 we establish a result on the associative enveloping algebra of $\mathcal{A}$ depending on values of $\lambda$ and $\beta$.

\begin{Th}
	Let $\mathcal{A}$ be the associative subalgebra of $\End\big(\mathcal{B'}(\lambda ,\beta)\big)$ with generators $\{l_o ,l_a, l_b \}$. Then the following statements hold:
	\begin{enumerate}
		\item[1.] If $\lambda = \beta = \frac{1}{2}$, then $\mathcal{A} = M_0$.
		\item[2.] If $\lambda \neq \beta, \lambda = 1 - \beta$, then $\mathcal{A} = M_1$.
		\item[3.] If $\lambda = \beta, \lambda \neq 1 - \beta$, then $\mathcal{A} = M_2$.
		\item[4.] If $\lambda \neq \beta, \lambda \neq 1 - \beta$, then $\mathcal{A} = M_3$.
	\end{enumerate}
\end{Th}

\begin{proof}
 Since $\lambda(1-\lambda) (E_{22}+E_{33})=l_o -l_o ^2 \in \mathcal{A}$,
${E_{11}}=l_o -\lambda (E_{22}+E_{33})\in \mathcal{A}$
and $\lambda E_{21}=l_a-l_a \cdot (E_{22}+E_{33})\in \mathcal{A}$, these give ${E_{22}+E_{33}}, E_{11},  E_{21}\in \mathcal{A}$.

Consider $x_1:=\frac{\lambda-\beta}{\lambda}E_{23}+\frac{\beta-1}{\lambda}E_{33}+E_{43}=l_a\cdot(E_{22}+E_{33})-(E_{22}+E_{33})\in \mathcal{A}$.
 Then $x_2:=\frac{\lambda-\beta}{\lambda}E_{23}+\frac{\beta-1}{\lambda}E_{33}=(E_{22}+E_{33})\cdot(\frac{\lambda-\beta}{\lambda}E_{23}
 +\frac{\beta-1}{\lambda}E_{33}+E_{43})\in \mathcal{A}$ and we obtain  ${E_{43}}=x_1-x_2 \in \mathcal{A}$.

Consider $x_3:=(l_b-\frac{\lambda-\beta}{\lambda}(E_{22}+E_{33}))\cdot (E_{22}+E_{33}) \in \mathcal{A}$
 and $\lambda E_{31}= l_b-\frac{\lambda-\beta}{\lambda}(E_{22}+E_{33})-x_3\in\mathcal{A}$.
 So we obtain ${E_{31}}\in \mathcal{A}$ and ${E_{41}}=E_{43}\cdot E_{31}\in \mathcal{A}$.

Moreover, ${E_{42}}=x_3- (E_{22}+E_{33})x_3\in \mathcal{A}$.

We have $ E_{33}=\frac{\lambda}{\lambda-1}(x_3-E_{42})\cdot x_2
\in \mathcal{A}$
 and $E_{22}=(E_{22}+E_{33})-E_{33}\in \mathcal{A}$.

Summarizing, we obtain $M_0 \subseteq \mathcal{A}$.

Furthermore, from $x_2$ and $x_3-E_{42}$, we obtain $\frac{\lambda-\beta}{\lambda}E_{23}\in \mathcal{A}$ and $\frac{\lambda+\beta-1}{\lambda}E_{32}\in\mathcal{A}$.
 Thus the following cases occur:
\begin{enumerate}
\item[\textbf{Case 1.}] Let $\lambda\neq \beta$ and $\lambda\neq 1-\beta$.

Then $E_{23}, E_{32}\in \mathcal{A}$ and $\mathcal{A}= M_3$.

\item[\textbf{Case 2.}] Let $\lambda= \beta$ and $\lambda\neq 1-\beta$.

Then we obtain $E_{32}\in \mathcal{A}$ and $ M_2 \subseteq \mathcal{A}$. However, $l_o,l_a,l_b \in M_2$ and therefore $\mathcal{A}= M_2$.

\item[\textbf{Case 3.}] Let $\lambda\neq \beta$ and $\lambda=1-\beta$.

Then we have $E_{23}\in \mathcal{A}$ and $ M_1 \subseteq \mathcal{A}$. Moreover, $l_o,l_a,l_b \in M_1$ and $\mathcal{A}= M_1$.

\item[\textbf{Case 4.}] Let $\lambda=\beta= 1-\beta=\frac12$.

Then we get $l_o,l_a,l_b \in M_0$ and $\mathcal{A}= M_0$.
\end{enumerate}
\end{proof}

\section{Isomorphisms of ABO-type blood  algebras}

In this section we analyze the conditions under which two  algebras  $\mathcal{B'}(\lambda ',\beta')$ and $\mathcal{B'}(\lambda ,\beta)$
with corresponding basis $\{o ',a',b',ab'\}$ and $\{o ,a,b,ab\} $ are isomorphic.

Let the isomorphism $\varphi \colon \mathcal{B'}(\lambda ',\beta')\to  \mathcal{B'}(\lambda ,\beta) $ be given by
\[\varphi(o ')=d_{11}o +d_{12}a+d_{13}b+d_{14}ab\]
\[\varphi(a')=d_{21}o +d_{22}a+d_{23}b+d_{24}ab\]
\[\varphi(b')=d_{31}o +d_{32}a+d_{33}b+d_{34}ab\]
\[\varphi(ab')=d_{41}o +d_{42}a+d_{43}b+d_{44}ab.\]

Since $\langle ab' \rangle =\Ann(\mathcal{B'}(\lambda ',\beta'))$ it is clear that $\varphi(ab')\in \Ann(\mathcal{B'}(\lambda ,\beta))$.

Therefore, \[0=\varphi(ab')\circ o = (d_{41}o +d_{42}a+d_{43}b+d_{44}ab)\circ o  =d_{41}o +\lambda d_{42}a+\lambda d_{43}b\]
and since $\lambda\neq 0$, we obtain $d_{41}=d_{42}=d_{43}=0$.

Considering the equalities \[\varphi(o ')\circ \varphi(o ')= \varphi(o '\circ o '), \, \varphi(a')\circ \varphi(a')=\varphi(a'\circ a'), \,\varphi(b')\circ\varphi(b')=\varphi(b'\circ b'),\]
\[ \varphi(o')\circ \varphi(a')= \varphi(o '\circ a'),\, \varphi(o ') \circ \varphi (b') =
\varphi(o '\circ b'),\, \varphi(a')\circ\varphi(b')=\varphi(a'\circ b'),\] and comparing the coefficients at
the corresponding basis elements $\{ o,a,b,ab\}$ we derive the system of equations
\[ (I):\left\{\begin{array}{rl}
d_{11}=& d_{11}^2\\
d_{12}=& d_{12}^2+2d_{11}d_{12}\lambda +2d_{12}d_{13} \frac{\lambda-\beta}{\lambda}\\
d_{13}=& d_{13}^2+2d_{11}d_{13}\lambda +2d_{12}d_{13} \frac{\lambda+\beta-1}{\lambda}\\
d_{14}=& 2d_{12}d_{13}
\end{array}
\right.\]

\[ (II):\left\{\begin{array}{rl}
d_{21}=& d_{21}^2\\
d_{22}=& d_{22}^2+2d_{21}d_{22}\lambda +2d_{22}d_{23}   \frac{\lambda-\beta}{\lambda}\\
d_{23}=& d_{23}^2+2d_{21}d_{23}\lambda +2d_{22}d_{23} \frac{\lambda+\beta-1}{\lambda}\\
d_{24}=& 2d_{22}d_{23}
\end{array}
\right.\]
\[(III): \left\{\begin{array}{rl}
d_{31}=& d_{31}^2\\
d_{32}=& d_{32}^2+2d_{31}d_{32}\lambda +2d_{32}d_{33}  \frac{\lambda-\beta}{\lambda}\\
d_{33}=& d_{33}^2+2d_{31}d_{33}\lambda +2d_{32}d_{33} \frac{\lambda+\beta-1}{\lambda}\\
d_{34}=& 2d_{32}d_{33}
\end{array}
\right.\]
\[ (IV):\left\{\begin{array}{rl}
\lambda'd_{21}=& d_{11}d_{21}\\
\lambda'd_{22}=& d_{12}d_{22} + (d_{11}d_{22}+d_{12}d_{21})\lambda+ (d_{12}d_{23}+d_{13}d_{22}) \frac{\lambda-\beta}{\lambda} \\
\lambda'd_{23}=&d_{13}d_{23} +(d_{11}d_{23}+d_{13}d_{21})\lambda+ (d_{12}d_{23}+d_{13}d_{22}) \frac{\lambda+\beta-1}{\lambda}\\
\lambda'd_{24}=& d_{12}d_{23}+d_{13}d_{22}
\end{array}
\right.\]
\[ (V):\left\{\begin{array}{rl}
\lambda'd_{31}=& d_{11}d_{31}\\
\lambda'd_{32}=& d_{12}d_{32} + (d_{11}d_{32}+d_{12}d_{31})\lambda+ (d_{12}d_{33}+d_{13}d_{32})  \frac{\lambda-\beta}{\lambda} \\
\lambda'd_{33}=&d_{13}d_{33} +(d_{11}d_{33}+d_{13}d_{31})\lambda+ (d_{12}d_{33}+d_{13}d_{32}) \frac{\lambda+\beta-1}{\lambda}\\
\lambda'd_{34}=& d_{12}d_{33}+d_{13}d_{32}
\end{array}
\right.\]
\[(VI):\left\{\begin{array}{ll}
\frac{\lambda'-\beta'}{\lambda'}d_{21}+\frac{\lambda'+\beta'-1}{\lambda'}d_{31}&= d_{21}d_{31} \\
\frac{\lambda'-\beta'}{\lambda'}d_{22}+\frac{\lambda'+\beta'-1}{\lambda'}d_{32}&= d_{22}d_{32} +
 (d_{21}d_{32}+d_{22}d_{31})\lambda+ (d_{22}d_{33}+d_{23}d_{32}) \frac{\lambda-\beta}{\lambda}  \\
\frac{\lambda'-\beta'}{\lambda'}d_{23}+\frac{\lambda'+\beta'-1}{\lambda'}d_{33}& =  d_{23}d_{33}
 +(d_{21}d_{33}+d_{23}d_{31})\lambda+ (d_{22}d_{33}+d_{23}d_{32}) \frac{\lambda+\beta-1}{\lambda}\\
  \frac{\lambda'-\beta'}{\lambda'}d_{24}+\frac{\lambda'+\beta'-1}{\lambda'}d_{34}+d_{44} &=  d_{22}d_{33}+d_{23}d_{32}.
\end{array}
\right.\]

From the first equation of system (I) we obtain $d_{11}=0$ or $d_{11}=1$.

If $d_{11}=0$ then first equations of (IV) and (V) yield $d_{21}=d_{31}=0$, which together with $d_{41}=0$ gives a contradiction for $\varphi$ being an isomorphism.

Therefore, $d_{11}=1$ and from the same equations we obtain $d_{21}=d_{31}=0$.

Let us denote $\Delta = d_{22}d_{33}-d_{23}d_{32}$. Note that due to the results above $\det [\varphi]=  d_{44}\cdot \Delta\neq 0$.

Multiplying equations (IV.2), (V.2) by $d_{33},d_{23}$, respectively, and subtracting,
 we obtain $\lambda'\Delta=d_{12}\Delta+\lambda\Delta+d_{13}  \frac{\lambda-\beta}{\lambda} \Delta$. Since $\Delta \neq 0$ we obtain
\[\lambda'-\lambda=d_{12}+d_{13}  \frac{\lambda-\beta}{\lambda}.\]

Analogously, multiplying equations (IV.3), (V.3) by $d_{32}, d_{22}$ respectively, and subtracting from the second one the first one, we get
\[\lambda'-\lambda=d_{13}+d_{12}  \frac{\lambda+\beta-1}{\lambda}.\]

Hence, $d_{13}=\frac{1-\beta}{\beta}d_{12}$.

 \begin{enumerate}
\item[\textbf{Case 1.}] Let $d_{12}\neq 0$.

Then $d_{13}\neq 0$. Note that equations (I.2) and (I.3) transform in the system of equations \eqref{E:ec1} for the value $\alpha=1$.
\[\left\{\begin{array}{rl}
d_{12}+ 2 \frac{\lambda-\beta}{\lambda}d_{13}= & 1-2\lambda\\
2 \frac{\lambda+\beta-1}{\lambda}d_{12}+d_{13} =& 1-2\lambda
\end{array} \right. .\]

It is known that this system does not have a solution if the determinant of the system is equal to zero.
 Therefore, assuming that $(\lambda, \beta)\not \in P$ (that is, the determinant is not zero) we obtain the solution
\[d_{12}= \frac{\lambda (1-2  \lambda)}{-3\lambda^2+4\beta^2+4\lambda-4\beta}  (2\beta-\lambda),\]
\[ d_{13}= \frac{\lambda (1-2  \lambda)}{-3\lambda^2+4\beta^2+4\lambda-4\beta} (2-2\beta-\lambda).\]
 Taking into account $d_{13}=\frac{1-\beta}{\beta}d_{12}$, we obtain  $\beta=\frac12$. Hence, $d_{12}=d_{13}=\frac{\lambda(1-2\lambda)}{1-3\lambda}$.

Subtracting from (IV.2) the equation (IV.3) and taking into account that $d_{12}=d_{13}, \beta=\frac12$, we obtain

\[\lambda'(d_{22}-d_{23})=d_{12}(d_{22}-d_{23})+\lambda(d_{22}-d_{23}).\]

Similarly, subtracting from (V.2) equation (V.3), we obtain \[\lambda'(d_{32}-d_{33})=d_{12}(d_{32}-d_{33})+\lambda(d_{32}-d_{33}).\]

Observe that both values $d_{22}-d_{23}$ and $d_{32}-d_{33}$ can not be simultaneously zero, since it contradicts to $\Delta =0$.
 Therefore, at least one of these values is non-zero. Then we obtain $\lambda'-\lambda=d_{12}$ and $\lambda'-\lambda=d_{12} (1+\frac{\lambda-\beta}{\lambda})$.
  Hence $\lambda=\beta=\frac{1}{2}$. However it implies $d_{12}=0$, which is a contradiction.

\item[\textbf{Case 2.}] Let $d_{12}=0$.

Then $d_{13}=0$ and $\lambda'=\lambda$. From (I.4) we obtain $d_{14}=0$.
Furthermore, (IV.4) and (V.4) yield $d_{24}=d_{34}=0$. Equations (II.4) and (III.4) lead to $d_{22}d_{23}=d_{32}d_{33}=0$.
 Together with condition $\Delta \neq 0$ we have the following possible subcases:

\begin{enumerate}
\item[\textbf{Case 2.1.}] Let $d_{23}=d_{32}=0$ and $ d_{22}\cdot d_{33}\neq 0$.

Systems (II) and (III) yield $d_{22}=d_{33}=1$.
Substituting these values into (VI.2), we obtain  $\frac{\lambda'-\beta'}{\lambda'}=\frac{\lambda-\beta}{\lambda}$, which implies $\beta'=\beta$.

\item[\textbf{Case 2.2.}] Let $d_{22}=d_{33}=0$ and $ d_{23}\cdot d_{32}\neq 0$.

Then from (II.3) and (III.2) we get $d_{23}=d_{32}=1$.
Substituting the obtained values into the system (VI), we derive  $\beta'=1-\beta$.
\end{enumerate}
Hence, $\mathcal{B'}(\lambda , \beta)\cong \mathcal{B'}(\lambda , 1-\beta)$ via the change of basis $(o ,a,b,ab)\mapsto (o ,b,a,ab)$.
\end{enumerate}

We conclude summarizing the results in the following

\begin{Th}
	Two distinct ABO-blood type algebras  $\mathcal{B'}(\lambda , \beta)$ and $\mathcal{B'}(\lambda ', \beta')$ are isomorphic if and only if $\lambda'=\lambda$ and $\beta'=1-\beta$.
\end{Th}

\section*{Acknowledgments}

The  authors were supported by Ministerio de Econom\'ia y Competitividad (Spain), grant MTM2013-43687-P (European FEDER support included).
The second and fourth  authors were also supported by Xunta de Galicia,
grant GRC2013-045 (European FEDER support included). The third named author was also supported by the Grant No.0251/GF3 of Education and Science Ministry of Republic of Kazakhstan.

\end{document}